\documentclass[11pt,a4paper]{article}

\usepackage[utf8]{inputenc}
\usepackage[T1]{fontenc}
\usepackage[english]{babel}
\usepackage{amsmath}
\usepackage{amsfonts}
\usepackage{amsthm}
\usepackage{amssymb}
\usepackage{makeidx}
\usepackage{graphicx}
\usepackage{lmodern}
\usepackage[left=2cm,right=2cm,top=2cm,bottom=2cm]{geometry}

\usepackage{color}
\usepackage{comment}
\usepackage[dvipsnames]{xcolor}
\usepackage{graphicx}
\usepackage{framed}
\usepackage{tikz}
\usepackage{calrsfs}
\DeclareMathAlphabet{\pazocal}{OMS}{zplm}{m}{n}

\usepackage{hyperref} 
\hypersetup{
	linktocpage = true,
	colorlinks = true,
	linkcolor = {RoyalBlue},
	urlcolor = {RoyalBlue},
	citecolor = {Green}}

\usepackage{fourier-orns}
\usepackage{mathtools}
\usepackage{relsize}
\usepackage{dsfont}
\usepackage{comment}
\usepackage{hyperref}
  
\newcommand{\R}{\mathbb{R}}

\renewcommand{\H}{\mathbb{H}}

\newcommand{\F}{\pazocal{F}}
\newcommand{\G}{\pazocal{G}}
\newcommand{\U}{\pazocal{U}}

\newcommand{\M}{\pazocal{M}}

\newcommand{\Lpazo}{\pazocal{L}}
\newcommand{\Ppazo}{\pazocal{P}}

\newcommand{\Cpazo}{\pazocal{C}}

\newcommand{\Lcal}{\mathcal{L}}
\newcommand{\Vcal}{\mathcal{V}}
\newcommand{\Pcal}{\mathcal{P}}
\newcommand{\Ccal}{\mathcal{C}}
\newcommand{\Scal}{\mathcal{S}}

\newcommand{\Id}{\textnormal{Id}}
\newcommand{\D}{\textnormal{D}}

\newcommand{\supp}{\textnormal{supp}}

\newcommand{\Div}{\textnormal{div}}

\newcommand{\textbn}[1]{\textnormal{\textbf{#1}}}

\newcommand{\vb}{\boldsymbol{v}}

\newcommand{\Bmu}{\boldsymbol{\mu}}

\newcommand{\INTDom}[3]{\int_{#2} #1 \textnormal{d} #3}
\newcommand{\INTSeg}[4]{\int_{#3}^{#4} #1 \textnormal{d} #2}

\newcommand{\Norm}[1]{\parallel \hspace{-0.1cm} #1 \hspace{-0.1cm} \parallel}

\newcommand{\tderv}[2]{\tfrac{\textnormal{d} #1}{ \textnormal{d} #2}}

\newtheorem{Def}{Definition}
\newtheorem{thm}{Theorem}
\newtheorem{prop}{Proposition}

\newtheorem{cor}{Corollary}

\newtheorem{taggedhypx}{Hypotheses}
\newenvironment{taggedhyp}[1]
    {\renewcommand\thetaggedhypx{#1}\taggedhypx}
    {\endtaggedhypx}

\title{On the Properties of the Value Function Associated to a Mean-Field Optimal Control Problem of Bolza Type\thanks{\textit{This material is based upon work supported by the Air Force Office of Scientific Research under award number FA9550-18-1-0254.}}}

\author{Beno\^it Bonnet\thanks{CNRS, IMJ-PRG, UMR 7586, Sorbonne Universit\'e, 4 place Jussieu,  75252  Paris,  France {\tt\small benoit.bonnet@imj-prg.fr}} \hspace{0.05cm} and H\'el\`ene Frankowska\thanks{CNRS, IMJ-PRG, UMR 7586, Sorbonne Université, 4 place Jussieu,  75252  Paris,  France. {\tt\small helene.frankowska@imj-prg.fr}}%
}

\begin{document}

\maketitle


\begin{abstract}
In this paper, we obtain several structural results for the value function associated to a mean-field optimal control problem of Bolza type in the space of measures. After establishing the sensitivity relations bridging between the costates of the maximum principle and metric superdifferentials of the value function, we investigate semiconcavity properties of this latter with respect to both variables. We then characterise optimal trajectories using set-valued feedback mappings defined in terms of suitable directional derivatives of the value function. 
\end{abstract}


\section{Introduction}

During the past decade, variational problems formulated on mean-field approximations of collective dynamics have gained a tremendous amount of steam. Indeed, the mathematical field of \textit{multi-agent systems} analysis has been for several years the stage on which modern theories such as mean-field control \cite{Bensoussan2013,Carmona2018} and mean-field games \cite{Huang2006,Lasry2007} are being developed.

The main rationale behind studying the former of these two fields of research can be heuristically summarised as follows. From a practical standpoint, it is usually easier to model collective dynamics by means of large families of \textit{microscopic} entities -- called the agents --, which evolution is prescribed e.g. by ODEs \cite{Choi2014} or time-dependent graphs \cite{Egerstedt2010}. However, such formulations often lead to numerically intractable problems, and rarely allow to capture global properties of the system in a meaningful way. For these reasons, an ever-expanding literature has been focusing on the development of mathematical tools for the investigation of adequate \textit{macroscopic approximations} of such systems, see e.g. \cite{LipReg,Cardaliaguet2010,CDLL,ControlKCS} and references therein. In this context, an intensive research effort has been directed towards the establishment of a theory of \textit{mean-field optimal control} \cite{PMPWassConst,SetValuedPMP,PMPWass,Cavagnari2020,Cavagnari2021,Fornasier2019}. This denomination usually refers to optimal control problems formulated in the so-called \textit{Wasserstein spaces} (see Definition \ref{def:Wass} below) of optimal transport \cite{AGS}, which constitute a convenient and natural framework for the generalisation of the classical geometric results of control theory to this setting.

\medskip

In this paper, we transpose to this class of problems a series of result that were originally derived in \cite{Cannarsa1991}, dealing with fine properties of the value function associated to a general family of smooth optimal control problems. With this aim, we first investigate the so-called \textit{sensitivity relations} -- originally established in \cite{Flemming1975} -- which provide a link between the costate variables stemming from the Pontryagin maximum principle \cite{Pontryagin} (``PMP'' for short) and the superdifferentials of the value function. Among their many possible applications, the sensitivity relations are very useful for expressing sufficient conditions to ensure the optimality of Pontryagin's extremals. We proceed by providing a list of conditions under which the value function is \textit{semiconcave} \cite{CannarsaS2004} along suitably chosen interpolating curves. This notion of regularity is indeed very rich, and appears in a wide number of problems originating both in control theory and in the calculus of variations, see for instance \cite{Capuzzo1984,Rifford2002}. Moreover in the context of Wasserstein spaces, it has been well established in the monograph \cite{AGS} that semiconcavity properties along interpolating curves are in fact the natural path leading to meaningful results on metric superdifferentials in the space of measures. By leveraging some of these results, it is then possible to characterise optimal feedbacks for mean-field optimal control problems, and to provide useful regularity properties for the latter.

The structure of the article is the following. In Section \ref{section:Preli}, we recollect notions of measure theory and optimal transport, and proceed by recalling several results which are more specific to mean-field optimal control problems in Section \ref{section:MFOCP}. In Section \ref{section:Sensibility}, we study the sensitivity relations and semiconcavity properties of the value function for a general Bolza problem in Wasserstein spaces. We then make use of these results to investigate optimal feedbacks and their regularity in Section \ref{section:Reg}.


\section{Preliminaries}
\label{section:Preli}

In this section, we recall preliminary notions of measure theory and optimal transport, for which we refer to \cite{AmbrosioFuscoPallara} and \cite{AGS} respectively.


\subsection{Optimal transport and calculus in Wasserstein spaces}

In the sequel, $\Pcal(\R^d)$ denotes the space of Borel probability measures over $\R^d$, endowed with the standard \textit{narrow topology} (see \cite[Chapter 5]{AGS}). Given $p \in [1,+\infty)$, let $\Pcal_p(\R^d)$ be the subset of probability measure which $p$-momentum $\M_p(\mu) := \big(\INTDom{|x|^p}{\R^d}{\mu(x)}\big)^{1/p}$ is finite, and $\Pcal_c(\R^d)$ be that of measures with compact support. For $\mu \in \Pcal(\R^d)$, we denote by $L^p(\R^d,\R^d;\mu)$ the Banach space of maps from $\R^d$ into itself that are $p$-summable with respect to $\mu$. When $p = 2$, the latter is a Hilbert space endowed with the scalar product
\begin{equation*}
\langle \zeta,\xi \rangle_{L^2(\mu)} := \INTDom{\langle \zeta(x),\xi(x)\rangle}{\R^d}{\mu(x)}
\end{equation*}
defined here for any $\zeta,\xi \in L^2(\R^d,\R^d;\mu)$. We will also denote by $\Lcal^1$ the standard Lebesgue measure over $\R$, and consider the associated spaces $L^p(I,\R_+)$ of $p$-summable maps from an interval $I \subset \R$ into $\R_+$.

Given a Borel map $f : \R^d \rightarrow \R^d$, we define the \textit{pushforward} $f_{\#} \mu \in \Pcal(\R^d)$ of $\mu$ through $f$ as the measure
$
f_{\#} \mu(B) := \mu(f^{-1}(B)),
$
for every Borel set $B \subset \R^d$.

\begin{Def}[Transport plans]
The set of \textit{transport plans} between two measures $\mu,\nu \in \Pcal(\R^d)$ -- denoted by $\Gamma(\mu,\nu)$ -- is the collection of elements $\gamma \in \Pcal(\R^d \times \R^d)$ such that $\pi^1_{\#} \gamma = \mu$ and $\pi^2_{\#} \gamma = \nu$, where $\pi^1,\pi^2 : \R^d \times \R^d \rightarrow \R^d$ represent the projection operator onto the first and second factor respectively.
\end{Def}

In the sequel, we will also need the following simplified version of \cite[Theorem 5.3.1]{AGS}.

\begin{thm}[Barycentric projection]
\label{thm:Disintegration}
Let $p \in [1,+\infty)$, $\mu,\nu \in \Pcal_p(\R^d)$ and $\gamma \in \Gamma(\mu,\nu)$. Then, there exists a unique map $\bar{\gamma} \in L^p(\R^d,\R^d;\mu)$ -- called the \textit{barycentric projection} of $\gamma$ onto $\pi^1_{\#} \gamma = \mu$ --, such that
\begin{equation*}
\INTDom{\langle y , \F(x) \rangle}{\R^{2d}}{\gamma(x,y)} = \langle \bar{\gamma} , \F \rangle_{L^2(\mu)},
\end{equation*}
for every $\F \in L^{\infty}(\R^d,\R^d;\mu)$.
\end{thm}

\begin{Def}[Wasserstein distances]
\label{def:Wass}
The quantity
\begin{equation*}
W_p(\mu,\nu) := \min_{\gamma \in \Gamma(\mu,\nu)} \Big( \INTDom{|x-y|^p}{\R^{2d}}{\gamma(x,y)} \Big)^{1/p},
\end{equation*}
defines a distance between any two measures $\mu,\nu \in \Pcal_p(\R^d)$, and we denote by $\Gamma_o(\mu,\nu)$ the set of transport plans for which the minimum is attained.
\end{Def}

The Wasserstein space of order $p \in [1,+\infty)$ is defined as the metric space $(\Pcal_p(\R^d),W_p)$. In the sequel, we will often consider the (non complete) metric space $(\Pcal_c(\R^d),W_p)$ of compactly supported measures equipped with the $W_p$-metric. When $p =2$, it has been known since \cite{AGS,Otto2001} that $(\Pcal_2(\R^d),W_2)$ could be endowed -- at least formally -- with a differentiable manifold-like structure. We present below a revisited definition of the corresponding calculus, tailored for $(\Pcal_c(\R^d),W_2)$, which is borrowed from \cite{SetValuedPMP}.  Given $R>0$, we will use the notation $B_R(\mu) := \cup_{x \in \supp(\mu)} B(x,R)$.

\vspace{0.05cm}

\begin{Def}[Local differentials]
\label{def:DiffWass}
A functional $\phi : \Pcal_2(\R^d) \mapsto \R \cup \{ \pm \infty \}$ is \textit{locally differentiable} at $\mu \in \Pcal_c(\R^d)$ if and only if there exists $\nabla \phi(\mu) \in L^2(\R^d,\R^d;\mu)$ such that for all $R>0$ and $\nu \in \Pcal(B_R(\mu))$, it holds
\begin{equation*}
\phi(\nu) = \phi(\mu) + \INTDom{\langle \nabla \phi(\mu),y-x\rangle}{\R^{2d}}{\gamma(x,y)} + o_R(W_2(\mu,\nu)),
\end{equation*}
for every $\gamma \in \Gamma_o(\mu,\nu)$. We also say that $\phi(\cdot)$ is \textit{locally continuously differentiable} if the application $(\mu,x) \in \Pcal_c(\R^d) \times \R^d \mapsto \nabla \phi(\mu)(x) \in \R^d$ is continuous.
\end{Def}

In \cite{SetValuedPMP}, we proved the following chain rule property.

\begin{prop}[Chain rule]
\label{prop:Chainrule}
Consider elements $\mu \in \Pcal_c(\R^d)$, $\F \in L^{\infty}(\R^d,\R^d;\mu)$ and let $\phi : \Pcal_2(\R^d) \rightarrow \R \cup \{ \pm \infty\}$ be locally differentiable at $\mu$. Then, it holds
\begin{equation*}
\phi((\Id+\epsilon \F)_{\#} \mu) = \phi(\mu) + \epsilon \langle \nabla \phi(\mu) , \F \rangle_{L^2(\mu)} + o(\epsilon).
\end{equation*}
\end{prop}


\subsection{Non-local continuity equations}

Consider the \textit{non-local continuity equation} in a Wasserstein space
\begin{equation}
\label{eq:NonLocalCE}
\partial_t \mu(t) + \Div_x(v(t,\mu(t))\mu(t)) = 0, \quad  \mu(0) = \mu^0,
\end{equation}
where the \textit{non-local velocity-field} $v : [0,T] \times \Pcal_c(\R^d) \times \R^d \rightarrow \R^d$ satisfies the following assumptions.

\begin{taggedhyp}{\textbn{(CE)}}
\label{hyp:CE}
Suppose that for any compact set $K \subset \R^d$, the following holds.
\begin{enumerate}
\item[$(i)$] The mapping $t \mapsto v(t,\mu,x) \in \R^d$ is $\Lcal^1$-measurable with respect to $t \in [0,T]$ for all $(\mu,x) \in \Pcal_c(\R^d) \times \R^d$, and there exists $m(\cdot) \in L^1([0,T],\R_+)$ such that
\begin{equation*}
|v(t,\mu,x)| \leq m(t) \Big( 1+|x| + \M_1(\mu) \Big).
\end{equation*}
\item[$(ii)$] There exists a map $\ell_K(\cdot) \in L^1([0,T],\R_+)$ such that
\begin{equation*}
|v(t,\nu,y) - v(t,\mu,x)| \leq \ell_K(t) \big( W_1(\mu,\nu) + |x-y| \big)
\end{equation*}
for $\Lcal^1$-almost every $t \in [0,T]$, any $\mu,\nu \in \Pcal(K)$ and all $x,y \in K$.
\end{enumerate}
\end{taggedhyp}

Under these Cauchy-Lipschitz assumptions, we have the following well-posedness result for \eqref{eq:NonLocalCE} (see e.g. \cite{ContInc}).

\begin{thm}[Well-posedness of \eqref{eq:NonLocalCE}]
Let $r >0$, $\mu^0 \in \Pcal(B(0,r))$ and suppose that hypotheses \ref{hyp:CE} hold. Then, there exists a unique distributional solution $\mu(\cdot)$ of \eqref{eq:NonLocalCE}, which also satisfies
\begin{equation*}
\supp(\mu(t)) \subset B(0,R_r), \quad W_1(\mu(t),\mu(\tau)) \leq \INTSeg{m_r(s)}{s}{\tau}{t},
\end{equation*}
for all $0 \leq \tau \leq t \leq T$, with $R_r >0$ and $m_r(\cdot) \in L^1([0,T],\R_+)$ depending on $r > 0$ and $m(\cdot)$. Moreover, the curve $\mu(\cdot)$ can be represented explicitly as
\begin{equation*}
\mu(t) = \Phi_{(\tau,t)}[\mu(\tau)](\cdot)_{\#} \mu(\tau),
\end{equation*}
for every $\tau,t \in [0,T]$, where $(\Phi_{(\tau,t)}[\mu](\cdot))_{\tau,t \in [0,T]}$ denotes the \textit{non-local flow} starting from $\mu \in \Pcal_c(\R^d)$ at time $\tau \in [0,T]$, defined as the unique solution of
\begin{equation}
\label{eq:Semigroup}
\left\{
\begin{aligned}
\partial_t \Phi_{(\tau,t)}[\mu](x) & = v \Big( t ,\Phi_{(\tau,t)}[\mu](\cdot)_{\#} \mu , \Phi_{(\tau,t)}[\mu](x)\Big), \\
\Phi_{(\tau,\tau)}[\mu](x) & = x,
\end{aligned}
\right.
\end{equation}
for all $x \in \R^d$.
\end{thm}


\section{Mean-field optimal control}
\label{section:MFOCP}

In this section, we focus more in details on the following mean-field optimal control problem of Bolza type
\begin{equation*}
(\Ppazo) ~
\left\{
\begin{aligned}
\min_{u(\cdot) \in \U} & \bigg[ \INTSeg{L(\mu(t),u(t))}{t}{0}{T} + \varphi(\mu(T)) \bigg], \\
\text{s.t.} ~ & \left\{
\begin{aligned}
& \partial_t \mu(t) + \Div_x(v(t,\mu(t),u(t))\mu(t)) = 0, \\
& \mu(0) = \mu^0.
\end{aligned}
\right.
\end{aligned}
\right.
\end{equation*}
Here $T>0$, $v : [0,T] \times \Pcal_c(\R^d) \times U \times \R^d \rightarrow \R^d$ is a controlled velocity-field, the set of admissible controls is $\U := L^2([0,T],U)$ where $U \subset \R^m$ is a closed set, and $\mu^0 \in \Pcal_c(\R^d)$ is a fixed initial datum. In the present paper, we make the following assumptions on $(\Ppazo)$.

\begin{taggedhyp}{\textbn{(OCP)}}
\label{hyp:OCP}
Suppose that for any compact set $K \subset \R^d$, the following holds.
\begin{enumerate}
\item[$(i)$] The map $u \in U \mapsto v(t,\mu,u,x) \in \R^d$ is continuous, and there exist constants $m,\ell_K > 0$ such that
\begin{equation*}
\left\{
\begin{aligned}
& |v(t,\mu,u,x)| \leq m \Big( 1+|x| + \M_1(\mu) \Big)(1+|u|^2), \\
& |v(\tau,\nu,u,y)- v(t,\mu,u,x)| \\
& \hspace{0.45cm} \leq \ell_K \big( |t-\tau| + W_1(\mu,\nu) + |x-y|\big)\big(1+|u|^2),
\end{aligned}
\right.
\end{equation*}
for all $t,\tau \in [0,T]$, any $\mu,\nu \in \Pcal(K)$, every $u \in U$ and each $x,y \in K$.
\item[$(ii)$] The map $u \in U \mapsto L(\mu,u) \in \R_+$ is lower semicontinuous, and there exist constants $l^1,\Cpazo^1,\Lpazo_K^1 > 0$ such that $L(\mu,u) \geq \Cpazo^1 |u|^2$ and
\begin{equation*}
\left\{
\begin{aligned}
& L(\mu,u) \leq l^1(1+\M_1(\mu))(1+|u|^2), \\
& |L(\nu,u) - L(\mu,u)| \leq \Lpazo_K^1 W_1(\mu,\nu)(1+|u|^2),
\end{aligned}
\right.
\end{equation*}
for any $\mu,\nu \in \Pcal(K)$ and every $u \in U$.
\item[$(iii)$] There exist $\Cpazo^2 \in \R$ and $\Lpazo^2_K > 0$ such that
\begin{equation*}
\varphi(\mu) \geq \Cpazo^2 \quad \text{and} \quad |\varphi(\nu) - \varphi(\mu)| \leq \Lpazo_K^2 W_1(\mu,\nu),
\end{equation*}
for any $\mu,\nu \in \Pcal(K)$.
\item[$(iv)$] The maps $(x,\mu) \mapsto v(t,\mu,u,x),L(u,\mu),\varphi(\mu)$ are locally continuously differentiable in the sense of Definition \ref{def:DiffWass}.
\end{enumerate}
\end{taggedhyp}

While the assumptions listed above are not minimal for the developments of this section (see e.g. more general conditions in \cite[Section 4.2]{SetValuedPMP}), they will be necessary to establish several results in Section \ref{section:Sensibility} and Section \ref{section:Reg}.

\begin{thm}[Existence]
\label{thm:Existence}
Suppose that \ref{hyp:OCP} hold and that for any non-empty compact $K \subset \R^d$ and every $(t,\mu) \in [0,T] \times \Pcal_c(\R^d)$, the restricted sets
\begin{equation*}
\begin{aligned}
\Big\{ (v(t,\mu,u)_{\vert K},L(\mu,u)+r ) ~\text{s.t.}~ u \in U,~ & r\geq 0 \Big\} \subset C^0(K,\R^d) \times \R,
\end{aligned}
\end{equation*}
are closed and convex. Then, there exists an optimal trajectory-control pair $(\mu^*(\cdot),u^*(\cdot))$ for $(\Ppazo)$.
\end{thm}

\begin{proof}
Under hypothesis \ref{hyp:OCP}-$(ii)$, there exists for any minimising sequence $(\mu^*_N(\cdot),u^*_N(\cdot))$ a constant $\Cpazo_U > 0$ such that
\begin{equation*}
\sup_{N \geq 1} \INTSeg{|u^*_N(t)|^2}{t}{0}{T} \leq \Cpazo_U.
\end{equation*}
The result then follows up to minor modifications in the proofs of \cite[Theorem 6 and Theorem 7]{ContInc}.
\end{proof}

In the sequel, we will assume that $(\Ppazo)$ admits such an optimal pair $(\mu^*(\cdot),u^*(\cdot))$. By a simple adaptation of the arguments of \cite{SetValuedPMP}, it can be shown that the latter satisfies the following Pontryagin-type optimality conditions.

\begin{thm}[Pontryagin Maximum Principle]
\label{thm:PMP}
Let $(\mu^*(\cdot),u^*(\cdot))$ be an optimal pair for $(\Ppazo)$. Then, there exists a unique state-costate curve $\nu^* : [0,T] \mapsto \Pcal_c(\R^{2d})$ which satisfies the following.
\begin{enumerate}
\item[$(i)$] It solves the Hamiltonian continuity equation
\begin{equation*}
\left\{
\begin{aligned}
& \partial_t \nu^*(t) + \Div_{x,r} \big( \nabla_{\nu} \H(t,\nu^*(t),u^*(t)) \nu^*(t) \big) = 0, \\
& \pi^1_{\#} \nu^*(t) = \mu^*(t) \hspace{1.3cm} \text{for all times $t \in [0,T]$}, \\
& \nu^*(T) = (\Id,-\nabla \varphi(\mu^*(T)))_{\#} \mu^*(T),
\end{aligned}
\right.
\end{equation*}
where the Hamiltonian $\H$ is defined by
\begin{equation}
\label{eq:HamiltonianDef}
\H(t,\nu,u) := \INTDom{\langle r , v(t,\mu,u,x) \rangle}{\R^{2d}}{\nu(x,r)} - L(\mu,u),
\end{equation}
with $\mu := \pi^1_{\#} \nu$ and $\nabla_{\nu} \H(t,\nu^*(t),u^*(t))$ being its gradient given in the sense of Definition \ref{def:DiffWass}.
\item[$(ii)$] The maximisation condition
\begin{equation}
\label{eq:HamiltonianMax}
\H(t,\nu^*(t),u^*(t)) = \max_{u \in U} \, \H(t,\nu^*(t),u),
\end{equation}
holds for $\Lcal^1$-almost every $t \in [0,T]$.
\end{enumerate}
\end{thm}


\section{Properties of the value function}
\label{section:Sensibility}

In this section, we leverage some of the tools introduced above  to  study fine properties of the \textit{value function}
\begin{equation*}
\begin{aligned}
\Vcal(\tau,\mu&) := \\
& \left\{
\begin{aligned}
\min_{u(\cdot) \in \U} & \bigg[ \INTSeg{L(\mu(t),u(t))}{t}{\tau}{T} + \varphi(\mu(T)) \bigg], \\
\text{s.t.} ~ & \left\{
\begin{aligned}
& \partial_t \mu(t) + \Div_x(v(t,\mu(t),u(t))\mu(t)) = 0, \\
& \mu(\tau) = \mu,
\end{aligned}
\right.
\end{aligned}
\right.
\end{aligned}
\end{equation*}
associated to $(\Ppazo)$, defined for $(\tau,\mu) \in [0,T] \times \Pcal_c(\R^d)$. A first key observation is that under our working assumptions, the mapping $\Vcal(\cdot,\cdot)$ is absolutely continuous in its first argument and Lipschitz in the second one.

\begin{prop}[Regularity]
\label{prop:LipReg}
Suppose that hypotheses \ref{hyp:OCP} hold, and let $r > 0$. Then, there exist $\M_r(\cdot) \in L^1([0,T],\R_+)$ and $\Lpazo_r > 0$ such that
\begin{equation*}
|\Vcal(\tau_2,\mu_2) - \Vcal(\tau_1,\mu_1)| \leq \INTSeg{\M_r(t)}{t}{\tau_1}{\tau_2} + \Lpazo_r W_1(\mu_1,\mu_2),
\end{equation*}
for every $\tau_1,\tau_2 \in [0,T]$ and all $\mu_1,\mu_2 \in \Pcal(B(0,r))$.
\end{prop}

\begin{proof}
This result follows from the Lipschitz dependence of solutions of non-local continuity equations on their initial data (see \cite{ContInc}) together with the sublinearity  and Lipschitz regularity of the cost functionals.
\end{proof}

\smallskip

In what follows, we will first investigate the so-called sensitivity relations involving the costates of the PMP of Theorem \ref{thm:PMP}, and then semiconcavity properties of $\Vcal(\cdot,\cdot)$. Given $u(\cdot) \in \U$, we will use the notation $(\Phi^u_{(\tau,t)}[\cdot](\cdot))_{\tau,t \in [0,T]}$ for the non-local flows generated by $(t,\mu,x) \mapsto v(t,\mu,u(t),x) \in \R^d$ defined as in \eqref{eq:Semigroup}.


\subsection{Sensitivity analysis and Pontryagin costates}

In our context, the sensitivity relations link the state-costate curves $\nu^*(\cdot)$ of the maximum principle and the \textit{Hadamard superdifferential} of the value function.

\begin{Def}[Hadamard superdifferentials]
Let $(\tau,\mu) \in [0,T] \times \Pcal_c(\R^d)$. Then, we say that a pair $(\delta,\xi) \in \R \times L^2(\R^d,\R^d;\mu)$ belongs to the Hadamard superdifferential $\eth^+ \Vcal(\tau,\mu)$ of $\Vcal(\cdot,\cdot)$ at $(\tau,\mu)$ provided that for all $(h,\F) \in \R \times L^{\infty}(\R^d,\R^d;\mu)$
\begin{equation*}
\begin{aligned}
& \limsup_{\substack{\epsilon \rightarrow 0^+ \\ \tau + \epsilon h \in [0,T]}} \Big[ \frac{\Vcal(\tau + \epsilon h, (\Id + \epsilon \F)_{\#} \mu) - \Vcal(\tau,\mu)}{\epsilon} \Big] \\
& \hspace{5cm} \leq \delta h + \langle \xi,\F\rangle_{L^2(\mu)}.
\end{aligned}
\end{equation*}
\end{Def}
\vspace{2mm}

On vector spaces, Hadamard superdifferentials are usually defined in terms of contingent directional derivatives. Yet, we showed in Proposition \ref{prop:LipReg} above that under \ref{hyp:OCP}, the map $\Vcal(t,\cdot)$ is locally Lipschitz over $\Pcal_c(\R^d)$ with local Lipschitz constants independent of $t \in [0,T]$. This allows us to use the simpler   upper  Dini directional derivatives to define Hadamard superdifferentials.

\begin{thm}[Hadamard-type sensitivity]
\label{thm:Sensitivity}
Suppose that hypotheses \ref{hyp:OCP} hold. Let $(\mu^*(\cdot),u^*(\cdot))$ be a solution of $(\Ppazo)$ and $\nu^*(\cdot)$ be the corresponding state-costate curve given by Theorem \ref{thm:PMP}. Then, it holds
\begin{equation*}
\big( \H(t,\nu^*(t),u^*(t)) , -\bar{\nu}^*(t) \big) \in \eth^+ \Vcal(t,\mu^*(t)),
\end{equation*}
for $\Lcal^1$-almost every $t \in [0,T]$, where the map $\bar{\nu}^*(t) \in L^2(\R^d,\R^d;\mu^*(t))$ denotes the \textit{barycentric projection} of $\nu^*(t)$ onto its first marginal $\pi^1_{\#} \nu^*(t) = \mu^*(t)$.
\end{thm}

\begin{proof}
The proof of this result is fairly long and technical, and will thus be published in greater details elsewhere. For the convenience of the reader, we sketch the main steps of the underlying arguments.

Given elements $(\tau,\epsilon) \in [0,T] \times \R_+$ and $(h,\F) \in \R \times L^{\infty}(\R^d,\R^d;\mu^*(\tau))$, the perturbed curve $\tilde{\mu}_{\epsilon}(\cdot)$ defined by
\begin{equation*}
\tilde{\mu}_{\epsilon}(t) := \Big( \Phi^{u^*}_{(\tau+\epsilon h,t)}[(\Id+\epsilon \F)_{\#}\mu^*
(\tau)] \circ (\Id+\epsilon \F) \Big)_{\raisebox{4pt}{$\scriptstyle{\#}$}}\mu^*(\tau),
\end{equation*}
for all times $t \in [\tau,T]$ satisfies
\begin{equation*}
\begin{aligned}
\Vcal(\tau+\epsilon h , (\Id+&\epsilon \F)_{\#} \mu^*(\tau)) \\
& \leq \INTSeg{L(\tilde{\mu}_{\epsilon}(t),u^*(t))}{t}{\tau+\epsilon h}{T} + \varphi(\tilde{\mu}_{\epsilon}(T)),
\end{aligned}
\end{equation*}
by definition of $\Vcal(\cdot,\cdot)$. Following technical linearisation arguments in the spirit of \cite{PMPWassConst,PMPWass}, one can then show
\begin{equation*}
\begin{aligned}
& \Phi^{u^*}_{(\tau+\epsilon h,t)}[(\Id+\epsilon \F)_{\#}\mu^*
(\tau)](x+\epsilon \F(x)) \\
& \hspace{1.3cm} = \Phi^*_{(\tau,t)}(x) + \epsilon w_{h,\F} \big(t,\Phi^*_{(t,\tau)}(x)\big) +o_{\tau,t,x}(\epsilon),
\end{aligned}
\end{equation*}
with $\sup_{(\tau,t,x)}|o_{\tau,t,x}(\epsilon)| = o(\epsilon)$ and where we denoted $\Phi^*_{(\tau,t)}(x) := \Phi^{u^*}_{(\tau,t)}[\mu^*(\tau)](x)$. Therein, $t \in [\tau,T] \mapsto w_{h,\F}(t,y)$ solves a suitable linearised Cauchy problem. Using the chain rule of Proposition \ref{prop:Chainrule} along with the optimality of $(\mu^*(\cdot),u^*(\cdot))$, the following inequality
\begin{equation*}
\begin{aligned}
& \tfrac{1}{\epsilon} \Big( \Vcal(\tau+\epsilon h , (\Id+\epsilon \F)_{\#} \mu^*(\tau)) - \Vcal(\tau,\mu^*(\tau)) \Big) \\
& \leq \INTSeg{\langle \nabla_{\mu} L(\mu^*(t),u^*(t)) , w_{h,\F}(t,\Phi^*_{(t,\tau)}(\cdot)) \rangle_{L^2(\mu^*(\tau))}}{t}{\tau}{T} \\
& \hspace{0.4cm} + \langle \nabla\varphi(\mu^*(T)), w_{h,\F}(T,\Phi_{(T,\tau)}^*(\cdot)) \rangle_{L^2(\mu^*(T))} \\
&\hspace{0.4cm} - h L(\mu^*(\tau),u^*(\tau)) + o(1),
\end{aligned}
\end{equation*}
holds for every $\epsilon > 0$ provided that $\tau \in [0,T]$ belongs to a suitable subset of full $\Lcal^1$-measure. Upon performing lengthy computations involving the dynamics of the state-costate curve $\nu^*(\cdot)$ introduced in Theorem \ref{thm:PMP}, this inequality can be further reformulated as
\begin{equation}
\label{eq:Est1}
\begin{aligned}
& \tfrac{1}{\epsilon} \Big( \Vcal(\tau+\epsilon h , (\Id+\epsilon \F)_{\#} \mu^*(\tau)) - \Vcal(\tau,\mu^*(\tau)) \Big) + o(1)\\
& \hspace{0.4cm} \leq \H(\tau,\nu^*(\tau),u^*(\tau)) \, h - \INTDom{\langle r , \F(x) \rangle}{\R^{2d}}{\nu^*(\tau)(x,r)} \\
& \hspace{0.4cm} \leq \H(\tau,\nu^*(\tau),u^*(\tau)) \, h - \langle \bar{\nu}^*(\tau),\F\rangle_{L^2(\mu^*(\tau))},
\end{aligned}
\end{equation}
for any $\epsilon > 0$ small enough, and where we used the barycentric projection of $\nu^*(\tau)$ defined as in Theorem \ref{thm:Disintegration}. We can then take the $limsup$ as $\epsilon \rightarrow 0^+$ in \eqref{eq:Est1} for $\Lcal^1$-almost every $t \in [0,T]$, which ends the proof.
\end{proof}

The main application of the sensitivity relations is to provide \textit{sufficient} optimality conditions for the PMP.

\begin{thm}[Sufficient optimality conditions]
Let $(\mu^*(\cdot),u^*(\cdot))$ be an admissible pair for $(\Ppazo)$ and $\nu^*(\cdot)$ be an associated state-costate curve satisfying the PMP of Theorem \ref{thm:PMP} and the sensitivity relation of Theorem \ref{thm:Sensitivity}. Then, the pair $(\mu^*(\cdot),u^*(\cdot))$ is optimal for $(\Ppazo)$.
\end{thm}

\begin{proof}
Observe that a pair $(\mu^*(\cdot),u^*(\cdot))$ is optimal for $(\Ppazo)$ if and only if
\begin{equation*}
\Vcal(\tau_2,\mu^*(\tau_2)) = \Vcal(\tau_1,\mu^*(\tau_1)) - \INTSeg{L(\mu^*(t),u^*(t))}{t}{\tau_1}{\tau_2},
\end{equation*}
for all $0 \leq \tau_1 \leq \tau_2 \leq T$.
By Proposition \ref{prop:LipReg}, the map $t \in [0,T] \mapsto \Vcal(t,\mu^*(t))$ is absolutely continuous, and by definition of the value function it necessary holds
\begin{equation}
\label{eq:ValueEst0}
\tderv{}{t} \Vcal(t,\mu^*(t)) \geq - L(\mu^*(t),u^*(t)),
\end{equation}
for $\Lcal^1$-almost every $t \in [0,T]$. Following \cite[Section 3.2]{SetValuedPMP}, it can be shown that for $\Lcal^1$-almost every $t \in [0,T]$
\begin{equation*}
\Phi^*_{(t,t+\epsilon)}(x) = x + \epsilon v(t,\mu^*(t),u^*(t),x) + o_{t,x}(\epsilon),
\end{equation*}
for every $\epsilon \in \R$, with $\INTSeg{\sup_x|o_{t,x}(\epsilon)|}{t}{0}{T} = o(\epsilon)$. This combined with \eqref{eq:Semigroup} and  Proposition \ref{prop:LipReg} then yields
\begin{equation}
\label{eq:ValueEst1}
\begin{aligned}
& \Vcal(t+\epsilon,\mu^*(t+\epsilon)) \\
& = \Vcal(t+\epsilon,(\Id+\epsilon v(t,\mu^*(t),u^*(t)))_{\#} \mu^*(t)) + o_t(\epsilon)
\end{aligned}
\end{equation}
with $\INTSeg{|o_t(\epsilon)|}{t}{0}{T} = o(\epsilon)$. By the sensitivity relation of Theorem \ref{thm:Sensitivity}, it further holds
\begin{equation}
\label{eq:ValueEst2}
\begin{aligned}
& \limsup_{\epsilon \rightarrow 0^+} \Big[ \tfrac{1}{\epsilon} \Big( \Vcal(t+\epsilon,(\Id + \epsilon v(t,\mu^*(t),u^*(t)))_{\#} \mu^*(t)) \\
& \hspace{5.8cm} - \Vcal(t,\mu^*(t)) \Big) \Big] \\
& \leq - \langle \bar{\nu}^*(t) , v(t,\mu^*(t),u^*(t)) \rangle_{L^2(\mu^*(t))} \hspace{-0.05cm} + \hspace{-0.05cm} \H(t,\nu^*(t),u^*(t)) \\
& = -L(\mu^*(t),u^*(t)),
\end{aligned}
\end{equation}
for $\Lcal^1$-almost every $t \in [0,T]$, where we used the definition \eqref{eq:HamiltonianDef} of the Hamiltonian. Upon combining \eqref{eq:ValueEst1} and \eqref{eq:ValueEst2}, we can thus recover
\begin{equation*}
\tderv{}{t} \Vcal(t,\mu^*(t)) \leq -L(\mu^*(t),u^*(t)),
\end{equation*}
for $\Lcal^1$-almost every $t \in [0,T]$, which together with \eqref{eq:ValueEst0} concludes the proof of our claim up to an integration over any time interval of the form $[\tau_1,\tau_2]$.
\end{proof}


\subsection{Regularity of the value function}

Our aim now is to study semiconcavity estimates for $\Vcal(\cdot,\cdot)$, taking inspiration from \cite[Chapter 9]{AGS}. We start our developments by recalling in which sense this notion is understood in Wasserstein spaces.

\begin{Def}[Interpolation semiconcavity]
A map $\phi : \Pcal_2(\R^d) \rightarrow \R \cup \{ \pm \infty \}$ is \textit{locally semiconcave} if it is finite on $\Pcal_c(\R^d) $,  and for every $r > 0$ there exists $\Cpazo_r > 0$ such that for any $\lambda \in [0,1]$,
\begin{equation*}
\begin{aligned}
(1-\lambda) \phi(\mu_1) + \lambda \phi(\mu_2) - & \phi \big( \Bmu^{1 \rightarrow 2}_{\lambda} \big) \\
& \leq \Cpazo_r \lambda(1-\lambda) W_{2,\Bmu}^2(\mu_1,\mu_2),
\end{aligned}
\end{equation*}
for any $\mu_1,\mu_2 \in \Pcal(B(0,r))$ and each $\Bmu \in \Gamma(\mu,\nu)$, where $\Bmu^{1\rightarrow2}_{\lambda} := \big((1-\lambda)\pi^1 + \lambda \pi^2 \big)_{\#} \Bmu$,  and \vspace{-0.1cm}
\begin{equation*}
W_{2,\Bmu}(\mu,\nu) := \Big( \INTDom{|x-y|^2}{\R^{2d}}{\Bmu(x,y)} \Big)^{1/2}.
\end{equation*}
\end{Def}

\vspace{0.1cm}

Below, we make the following additional assumptions on the data of $(\Ppazo)$.

\begin{taggedhyp}{\textbn{(R)}}
\label{hyp:R}
\begin{enumerate}
\item[$(i)$] For every $r > 0$, there exists a constant $\Cpazo_r^v > 0$ such that for all $\lambda \in [0,1]$, it holds
\begin{equation*}
\begin{aligned}
\big|(1-\lambda) v(t,\mu_1,u,x) & + \lambda v(t,\mu_2,u,y) \\
& - v(t,\Bmu^{1\rightarrow2}_{\lambda},u,(1-\lambda)x+\lambda y)| \\
& \hspace{-1.6cm} \leq \Cpazo_r^v \lambda(1-\lambda) \Big( W_{2,\Bmu}^2(\mu_1,\mu_2) + |x-y|^2\Big),
\end{aligned}
\end{equation*}
for all $(t,u) \in [0,T] \times U$, any $\mu_1,\mu_2 \in \Pcal(B(0,r))$ and each $x,y \in B(0,r)$.
\item[$(ii)$] The maps $\mu \in \Pcal_c(\R^d) \mapsto L(\mu,u),\varphi(\mu)$ are locally semiconcave, uniformly with respect to $u \in U$.
\end{enumerate}
\end{taggedhyp}

\begin{thm}[Semiconcavity]
\label{thm:Semiconcavity}
Suppose that hypotheses \ref{hyp:OCP} and \ref{hyp:R} hold. Then for every $r >0$, there exists $\Ccal_r > 0$ such that for any $\lambda \in [0,1]$, it holds
\begin{equation*}
\begin{aligned}
(1-\lambda) \Vcal(\tau_1,\mu_1) & + \lambda \Vcal(\tau_2,\mu_2) \\
& - \Vcal \big( (1-\lambda) \tau_1 + \lambda \tau_2 , \Bmu^{1 \rightarrow 2}_{\lambda} \big) \\
& \hspace{-0.4cm} \leq \Ccal_r \lambda(1-\lambda) \Big( |\tau_1-\tau_2|^2 + W_{2,\Bmu}^2(\mu_1,\mu_2) \Big),
\end{aligned}
\end{equation*}
for every $\tau_1,\tau_2 \in [0,T]$, any $\mu_1,\mu_2 \in \Pcal(B(0,r))$ and each $\Bmu \in \Gamma(\mu_1,\mu_2)$.
\end{thm}

\begin{proof}
The proof of this result is again fairly long and technical, and will be exposed in greater details in a subsequent publication. The core idea is to use the refined superposition results from \cite[Lemma 1]{ContInc} to prove that a norm-semiconcavity estimate similar to \ref{hyp:R}-$(i)$ holds for the non-local flows $(\Phi^{u}_{(\tau,t)}[\cdot](\cdot))_{\tau,t \in [0,T]}$. It is then possible to recover the local semiconcavity of $\Vcal(\tau,\cdot)$ for any fixed $\tau \in [0,T]$ by combining the latter estimates with hypotheses \ref{hyp:OCP}-$(ii)$,$(iii)$ and \ref{hyp:R}-$(ii)$.

To recover the semiconcavity with respect to both variables, one must perform a suitable rescaling of the time variable in the spirit of \cite[Section 5]{Cannarsa1991}, combined with the semiconcavity estimates already established with respect to the measure variable.
\end{proof}


\section{Optimal feedbacks}
\label{section:Reg}

By leveraging several results of section \ref{section:Sensibility}, it is possible to derive a general characterisation of optimal feedbacks for mean-field optimal control problems. Throughout this section, we assume that hypotheses \ref{hyp:OCP} hold.

We start by recalling the definition of an \textit{upper-semicontinuous} set-valued map (``u.s.c.'' for short).

\begin{Def}[Regular set-valued maps]
\label{def:USC}
A multifunction $\G : \Scal \rightrightarrows X$ from a metric space $(\Scal,d_{\Scal})$ into a Banach space $(X,\Norm{\cdot}_X)$ is upper-semicontinuous if for any $s \in \Scal$ and every neighbourhood $\Omega$ of $\G(s)$, there exists $\eta > 0$ such that $\G(s') \subset \Omega$ for every $s' \in B(s,\eta)$. In particular, if $\G(\cdot)$ is u.s.c. and single-valued, it is then continuous in the usual sense.
\end{Def}

In order to characterise optimal feedbacks for $(\Ppazo)$, we need to introduce two notions of inferior derivatives for functionals defined over the space of measures.

\begin{Def}[Two notions of lower-derivatives]
Let $\phi : \Pcal_2(\R^d) \rightarrow \R \cup \{ \pm \infty\}$ and $\mu \in \Pcal_c(\R^d)$,  $\phi(\mu) \neq \{ \pm \infty\}$. Given $\F \in C^0(\R^d,\R^d)$, we define the \textit{lower Dini derivative} of $\phi(\cdot)$ at $\mu$ in the direction $\F$ as
\begin{equation*}
\partial^-_{\D} \phi(\mu)(\F) := \liminf_{\epsilon \rightarrow 0^+} \Big[ \frac{\phi((\Id + \epsilon \F)_{\#} \mu) - \phi(\mu)}{\epsilon} \Big].
\end{equation*}
Given a parameter $R > 0$, we also define the \textit{regularised $R$-lower derivative} of $\phi(\cdot)$ at $\mu$ in the direction $\F$ as
\begin{equation*}
\partial^-_{o,R} \phi(\mu)(\F) := \liminf_{\substack{\nu \rightarrow \mu, \, \epsilon \rightarrow 0^+ \\ \nu \in \Pcal(B_R(\mu))}} \Big[ \frac{\phi((\Id + \epsilon \F)_{\#} \nu) - \phi(\nu)}{\epsilon} \Big].
\end{equation*}
\end{Def}

\begin{prop}[The case of semiconcave functions]
\label{prop:LowerDer}
Suppose that $\phi : \Pcal_2(\R^d) \rightarrow \R \cup \{ \pm \infty\}$ is locally semiconcave and let $\mu \in \Pcal_c(\R^d)$. Then, it holds that
\begin{equation*}
\partial_{\D}^- \phi(\mu)(\F) = \partial_{o,R}^- \phi(\mu)(\F),
\end{equation*}
for every $R>0$ and all $\F \in C^0(\R^d,\R^d)$.
\end{prop}

\begin{proof}
Adapt the arguments of \cite[Section 3]{Cannarsa1991}.
\end{proof}

In the following theorem, we combine these notions of non-smooth analysis to provide a general characterisation of optimal feedbacks for $(\Ppazo)$.

\begin{thm}[Optimal feedbacks]
\label{thm:Feedback}
An admissible pair $(\mu^*(\cdot),u^*(\cdot))$ is optimal for $(\Ppazo)$ if and only if
\begin{equation}
\label{eq:Feedback}
\begin{aligned}
v(t,\mu^*&(t),u^*(t)) \in \G(t,\mu^*(t)),
\end{aligned}
\end{equation}
for $\Lcal^1$-almost every $t \in [0,T]$, where we defined the sets
\begin{equation}
\label{eq:Gdef}
\begin{aligned}
\G(t,\mu) := \Big\{ v(t,\mu,u) ~\text{s.t.}~ \partial_{\D}^- \Vcal(& t,\mu) \big(1,v(t,\mu,u) \big) \\ 
& \leq \hspace{-0.1cm} -L(\mu,u) ~ \text{with $u \in U$}  \Big\},
\end{aligned}
\end{equation}
for all $(t,\mu) \in [0,T] \times \Pcal_c(\R^d)$. If in addition hypotheses \ref{hyp:R} hold and the velocities satisfy the stronger estimate
\begin{equation}
\label{eq:SublinNew}
|v(t,\mu,u,x)| \leq m \Big( 1+|x| + \M_1(\mu) \Big)(1+|u|), 
\end{equation}
then the restricted set-valued maps $\G_{|K} : [0,T] \times \Pcal(K) \rightrightarrows C^0(K,\R^d)$ are upper-semicontinuous for every compact set $K \subset \R^d$.
\end{thm}

\begin{proof}
The proof of the first part of this result -- namely that $(\mu^*(\cdot),u^*(\cdot))$ is optimal if and only if \eqref{eq:Feedback} holds $\Lcal^1$-almost everywhere in $[0,T]$ -- is quite long, and can be obtained by performing computations similar to those in the proof of Theorem \ref{thm:Sensitivity} above.

The regularity of the set-valued map $\G(\cdot,\cdot)$ under hypotheses \ref{hyp:R} can be established via the following arguments. By Theorem \ref{thm:Semiconcavity}, the value function $\Vcal(\cdot,\cdot)$ is locally semiconcave. Thus by Proposition \ref{prop:LowerDer}, the lower-derivative $\partial_{\D}^- \Vcal(t,\mu)(1,v(t,\mu,u))$ coincides with the regularised $R$-lower derivative $\partial^-_ {o,R}\Vcal(t,\mu)(1,v(t,\mu,u))$ for any $R>0$ and all $(t,\mu,u) \in [0,T] \times \Pcal_c(\R^d) \times U$. By definition of the latter together with the regularity assumptions \ref{hyp:OCP}-$(i),(ii)$ on the controlled velocity field and the running cost, it can be shown that for every $K \subset \R^d$, the graph of the restricted set-valued map $\G_{|K}(\cdot,\cdot)$ is closed. Moreover as a consequence of hypotheses \ref{hyp:OCP} and  \eqref{eq:SublinNew}, it can be checked that the images of $\G_{|K}(\cdot,\cdot)$ are contained in a compact subset of $C^0(K,\R^d)$. This concludes the proof since, by classical results in set-valued analysis (see e.g. \cite[Section 1.4]{Aubin1990}), every such multifunction with closed graph is upper-semicontinuous.
\end{proof}

In conclusion, the results of Theorem \ref{thm:Feedback} inform us that every optimal curve $\mu^*(\cdot)$ to a problem of the form $(\Ppazo)$ can be written as the solution of a \textit{continuity inclusion} defined in the sense of \cite{ContInc}.

\begin{cor}[Continuity inclusion]
A curve of measures $\mu^*(\cdot)$ is optimal for $(\Ppazo)$ if and only if it solves
\begin{equation}
\label{eq:ContinuityInc}
\partial_t \mu^*(t) \in - \Div_x(\G(t,\mu^*(t)\mu^*(t)),
\end{equation}
with $\G(\cdot,\cdot)$ as in \eqref{eq:Gdef}, namely if and only if there exists an $\Lcal^1$-measurable selection $t \in [0,T] \mapsto \vb^*(t) \in \G(t,\mu^*(t))$ such that the pair $(\mu^*(\cdot),\vb^*(\cdot))$ satisfies
\begin{equation*}
\partial_t \mu^*(t) + \Div_x(\vb^*(t)\mu^*(t)) = 0
\end{equation*}
in the sense of distributions.
\end{cor}

Even though the theory developed in \cite{ContInc} only accounts at the present moment for set-valued velocities that are Lipschitz with respect to the measure variable, the existence of solutions to \eqref{eq:ContinuityInc} is guaranteed whenever there exists at least one optimal solution to $(\Ppazo)$. In this context if we further assume that hypotheses \ref{hyp:R} hold, then for every compact set $K \subset \R^d$ the restricted set-valued maps $\G_{|K}(\cdot,\cdot)$ are upper-semicontinuous. The importance of this regularity result comes from the fact that set-valued maps with closed graphs are more convenient when  Euler constructions of solutions are involved. Also as previously mentioned, the applications $\G_{|K}(\cdot,\cdot)$ are continuous as soon as they are single-valued, e.g. when the value function is differentiable.


\bibliographystyle{plain}
{\footnotesize 
\bibliography{../../../ControlWassersteinBib}
}

\end{document}